\documentclass[10pt]{article}
\title{
Extending Linear Convergence of the Proximal Point Algorithm: The Quasar-convex Case} 
\author{Jos\'e M. M. de Brito \thanks{School of Science, Great Bay University; Great Bay Institute for Advanced Study,
Dongguan 523000, Guangdong Province,  People's Republic of China; and Instituto Federal do Piau\'{\i}, S\~ao Raimundo Nonato, Piau\'{\i}, Brazil. E-mail: jose.brito@ifpi.edu.br. ORCID-ID: 0000-0003-4362-0536
}  
\and
Felipe Lara\thanks{Instituto de Alta investigaci\'on (IAI), Universidad de
Tarapac\'a, Arica, Chile. E-mail: felipelaraobreque@gmail.com; flarao@academicos.uta.cl. 
Web: felipelara.cl, ORCID-ID: 0000-0002-9965-0921}
\and
Di Liu\thanks{Instituto Nacional de Matem\'atica Pura e Aplicada (IMPA), Rio de Janeiro, Brazil. E-mail: di.liu@impa.br, ORCID-ID: 0000-0002-0146-0055}
}

\usepackage{amsmath}
\usepackage{amsthm}
\usepackage{amssymb}
\usepackage{multicol}
\usepackage[active]{srcltx}
\usepackage{algorithm}
\usepackage{array}
\usepackage[dvipsnames,svgnames]{xcolor}
\usepackage{url}
\usepackage{tikz}
\usepackage{graphicx}
\usepackage{subcaption} 
\usepackage{pgfplots}
\usepackage{tabularx}
\usepackage{booktabs} 
\usepackage{xcolor}
\usepackage{ulem}  
\usepackage{comment}
\providecommand{\U}[1]{\protect \rule{.1in}{.1in}}
\newtheorem{theorem}{Theorem}

\newtheorem{corollary}[theorem]{Corollary}

\newtheorem{definition}[theorem]{Definition}
\newtheorem{example}[theorem]{Example}

\newtheorem{lemma}[theorem]{Lemma}

\newtheorem{proposition}[theorem]{Proposition}
\newtheorem{remark}[theorem]{Remark}

\pgfplotsset{compat=1.18}

\begin{document}


\maketitle

\begin{abstract}
\noindent This work investigates the properties of the proximity operator for quasar-convex functions and establishes the convergence of the proximal point algorithm to a global minimizer  with a particular focus on its convergence rate. In particular, we demonstrate: (i) the generated sequence is mi\-ni\-mi\-zing and achieves an $\mathcal{O}(\varepsilon^{-1})$ complexity rate for quasar-convex functions; (ii) under strong quasar-convexity, the sequence converges linearly and attains an $\mathcal{O}(\ln(\varepsilon^{-1}))$ complexity rate. These results extend known convergence rates from the (strongly) convex to the (strongly) quasar-convex setting. To the best of our knowledge, some findings are novel even for the special case of (strongly) star-convex functions. Numerical experiments corroborate our theoretical results. 

{\small}

\medskip


\medskip

\end{abstract}

\section{Introduction}

Let $h: \mathbb{R}^{n} \rightarrow \overline{\mathbb{R}} := \mathbb{R} \cup \{\pm \infty\}$ be a proper function. We are interested in solving the problem 
\begin{equation}\label{main:p}
 \min_{x \in \mathbb{R}^{n}} \, h(x),
\end{equation}
via the {\it Proximal Point Algorithm} (PPA henceforth) \cite{M1,M2,rock-1976}, a scheme which given $x^{0} \in \mathbb{R}^{n}$ and a sequence of positive parameters $\{\beta_{k}\}_{k}$, generates a sequence as follows:
\begin{equation}\label{ppa:intro}
 x^{k+1} \in {\rm argmin}_{x \in \mathbb{R}^{n}}\, \left( h(x) + \frac{1}{2 \beta_{k}} \lVert x^{k} - x\Vert^{2} \right), ~ \forall ~ k \in \mathbb{N}_{0} := \mathbb{N} \cup \{0\}.
\end{equation}

Despite its simplicity, PPA is a powerful tool in continuous optimization, as it allows us to simplify the original pro\-blem by solving better-conditioned subproblems. For instance, when \( h \) is convex, the function $x \mapsto h(x) + \frac{1}{2 \beta_{k}} \lVert x^{k} - x \rVert^{2}$ is strongly convex. This property makes solving the subproblems easier because proper, lower semicontinuous (lsc henceforth) strongly convex functions possess remarkable properties, such as a unique minimizer and quadratic growth. In fact, when \( h \) is convex, the sequence \(\{x^{k}\}_{k}\) generated by \eqref{ppa:intro} is a minimizing sequence for pro\-blem \eqref{main:p}. Specifically, $\{h(x^k)\}_k$ is decreasing, converges to the minimum (though not linearly), and achieves  a complexity of \(\mathcal{O}(\varepsilon^{-1})\). Moreover, if \( h \) is strongly convex, then \(\{x^{k}\}_{k}\) converges linearly to the unique solution with an improved complexity rate of \(\mathcal{O}(\ln(\varepsilon^{-1}))\). 
 
Due to these appealing properties, significant research efforts in recent de\-ca\-des have been devoted to extending the convergence guarantees of the PPA to broader function classes. Notable contributions include extensions to weakly convex and hypomonotone operators \cite{CP,IPS,PEN}; developments in DC pro\-gra\-mming \cite{AN,ABS,BB,SSC}; and studies of (strongly) quasiconvex functions \cite{GLM-1,Lara-9,PC}. However, these frameworks fail to fully preserve the desirable properties of the convex case. For instance, PPA achieves only sublinear convergence for weakly convex and DC functions. Although linear convergence is possible for strongly quasiconvex functions, there is no evidence that PPA even converges to a minimizer in the general quasiconvex (not strongly) case \cite{PC}.

This paper bridges this gap by presenting a comprehensive extension of PPA from convex to quasar-convex functions. We establish global convergence to a minimizer and provide a detailed analysis of its convergence rate. This extension is compelling for three key reasons:

First, the quasar-convex class is a strict generalization of fundamental convexity classes. It encompasses all convex functions with a nonempty solution set \cite{GGK,HMR}, all strongly convex functions as well as all (strongly) star-convex functions \cite{NP-2006}.

Second, quasar-convexity arises naturally in numerous nonconvex applications, including machine learning \cite{HMR,LV,LUW} and stochastic optimization \cite{Quantum,XZ}. Furthermore, its favorable geometry is also crucial for accelerating gradient-based me\-thods (see \cite{GGK,HLMV,HMR,HADR,HSS,WW} and references therein).

Third, as we demonstrate, quasar-convexity provides a natural modeling framework for novel problems in microeconomics, statistics, and machine learning that are beyond the reach of convex models (see Subsection \ref{subsec:3-1}).

\medskip

{\bf Our Contribution}: We study (strongly) quasar-convex functions with modulus $\gamma \geq 0$. To that end, first we establish sufficient conditions for (non\-smooth) functions to be (strongly) quasar-convex. In particular, we show that the $\ell_{p}$-regularization ($0<p<1$) \cite{Survey} and the CES (Constant Elasticity of Subs\-titution) \cite{ACMS}, Cobb-Douglas \cite{CD}, and Leontief \cite{Leon} production/utility functions are quasar-convex (and strongly quasar-convex on bounded convex sets). Second, we investigate the proximity operator of quasar-convex functions by proving that it is always nonempty (i.e., existence of iterates is gua\-ran\-teed) and demonstrating cases where it is not a singleton. Furthermore, we es\-ta\-blish a useful relationship between the minimizers of the proximity operator and the minimizers of the function. Third, we prove that the set of fixed points of the proximity operator coincides with the set of minimizers of a quasar-convex function, which provides the stopping criterion for PPA. Fourth, we implement PPA for quasar-convex functions ($\gamma \geq 0$) and show that the generated sequence is minimizing with  a complexity of $\mathcal{O}(\varepsilon^{-1})$ when $\gamma = 0$, matching the rate for convex functions \cite{Gu}. Fifth, for the special case of strongly quasar-convex functions ($\gamma > 0$), we prove that PPA converges linearly to the unique solution and achieves an $\mathcal{O}(\ln(\varepsilon^{-1}))$ complexity rate, exactly as in the strongly convex case. Finally, we present numerical experiments with nonconvex examples to demonstrate the potential of our results.

The paper is structured as follows: In Section \ref{sec:02}, we introduce preliminary concepts and basic definitions related to generalized convexity. In Section \ref{sec:03}, we study nonsmooth quasar-convex functions, analyze key properties of the pro\-xi\-mi\-ty operator, and present illustrative examples with applications in machine learning and economics. In Section \ref{sec:04}, we implement the proximal point algorithm (PPA) and analyze its convergence behavior and complexity rates for both strongly quasar-convex and quasar-convex cases. In Section \ref{sec:05}, we provide numerical experiments on nonconvex problems, demonstrating the practical implications of our theoretical results. Finally, in Section \ref{sec:06}, we conclude with a discussion of our findings and outline potential directions for future research.


\section{Preliminaries and Basic Definitions}\label{sec:02}

The inner product of $\mathbb{R}^{n}$ and the Euclidean norm are denoted by $\langle \cdot,\cdot \rangle$ and $\lVert \cdot \rVert$, respectively. The open ball with center at $x_{0} \in \mathbb{R}^{n}$ and radius $\varepsilon > 0$ is denoted by $\mathbb{B}(x_{0}, \varepsilon)$. We denote $\mathbb{R}_{+} := [0, + \infty[$ and $\mathbb{R}_{++} := \, ]0, + \infty[$, hence $\mathbb{R}^{n}_{+} := [0, + \infty[ \, \times \ldots \times [0, + \infty[$ and $\mathbb{R}_{++}^{n} := \, ]0, + \infty[ \, \times \ldots \times \, ]0, + \infty[$ ($n$ times), respectively. Let $K$ be a nonempty set in $\mathbb{R}^{n}$, its closure is denoted by $\overline{K}$ 
and its asymptotic (recession) cone by
$K^{\infty} := \{ u \in \mathbb{R}^{n}: ~ \exists ~ t_{k} \rightarrow + \infty, ~ \exists~ x^{k} \in K, ~ \frac{x^{k}}{t_{k}} \rightarrow u \}$. Furthermore, for any set $K \subset \mathbb{R}^{n}$, it follows from \cite[Proposition 2.1.2]{AT} that
$K^{\infty}=\{0\}$ if and only if $K$ is bounded.

Given any $x, y, z \in \mathbb{R}^{n}$ and any $\beta \in \mathbb{R}$, the following relations hold: 
\begin{align}
 & ~~~~~ \langle x - z, y - x \rangle= \frac{1}{2} \lVert z - y \rVert^{2} -
 \frac{1}{2} \lVert x - z \rVert^{2} - \frac{1}{2} \lVert y - x \rVert^{2}, 
 \label{3:points} \\
 & \lVert \beta x + (1-\beta) y \rVert^{2} = \beta \lVert x \rVert^{2} + (1 -
 \beta) \lVert y\rVert^{2} - \beta(1 - \beta) \lVert x - y \rVert^{2}. 
 \label{iden:1}
\end{align}

Given any extended-valued function $h: \mathbb{R}^{n} \rightarrow \overline{\mathbb{R}}$, the effective domain of $h$ is defined by ${\rm dom}\,h := \{x \in \mathbb{R}^{n}: h(x) < + \infty \}$. It is said that $h$ is proper if ${\rm dom}\,h$ is nonempty and $h(x) > - \infty$ for all $x \in \mathbb{R}^{n}$. The notion of properness is important when dealing with minimization pro\-blems.

We denote by ${\rm epi}\,h := \{(x,t) \in \mathbb{R}^{n} \times \mathbb{R}: h(x) \leq t\}$ the epigraph of $h$, by $S_{\lambda} (h) := \{x \in \mathbb{R}^{n}: h(x) \leq \lambda\}$ the sublevel set of $h$ at the height $\lambda \in \mathbb{R}$ and by ${\rm argmin}_{\mathbb{R}^{n}} h$ the set of all minimal points of $h$. A function $h$ is lsc at $\overline{x} \in \mathbb{R}^{n}$ if for any sequence $\{x_k\}_{k} \subset \mathbb{R}^{n}$ with $x_k \rightarrow \overline{x}$, we have $h(\overline{x}) \leq \liminf_{k \rightarrow + \infty} h(x_k)$. 

A function $h$ with convex domain is said to be
\begin{itemize}
 \item[$(a)$] (strongly) convex on ${\rm dom}\,h$, if there exists $\gamma \geq 0$ such that, for all $x, y \in {\rm dom}\,h$ and all $\lambda \in [0, 1]$, we have
 \begin{equation}\label{strong:convex}
  h(\lambda y + (1-\lambda)x) \leq \lambda h(y) + (1-\lambda) h(x) - \lambda (1 - \lambda) \frac{\gamma}{2} \lVert x - y \rVert^{2},
 \end{equation}

 \item[$(b)$] (strongly) quasiconvex on ${\rm dom}\,h$, if there exists $\gamma \geq 0$ such that, for all $x, y \in {\rm dom}\,h$ and all $\lambda \in [0, 1]$, we have
 \begin{equation}\label{strong:quasiconvex}
  h(\lambda y + (1-\lambda)x) \leq \max \{h(y), h(x)\} - \lambda(1 - \lambda) \frac{\gamma}{2} \lVert x - y \rVert^{2}.
 \end{equation}
\end{itemize} 
 \noindent A function is strongly convex (resp. quasiconvex) when $\gamma > 0$, and convex (resp. quasiconvex) when $\gamma = 0$. Hence, every (strongly) convex function is (strongly) quasiconvex, while the reverse statements do not hold in general (see \cite{CMA,HKS,Lara-9}).

Let \(K \subset \mathbb{R}^{n}\). A proper function \(h:\mathbb{R}^{n} \rightarrow \overline{\mathbb{R}}\) is said to be:
\begin{itemize}
 \item[$(i)$] 2-supercoercive on $K$, if
 \begin{equation}
  \liminf_{\substack{x\in K,\,\lVert x \rVert \rightarrow+ \infty}} \frac{h(x)}{\lVert x
  \rVert^{2}} >0,
 \end{equation}

 \item[$(ii)$] coercive on $K$, if
 \begin{equation}
  \lim_{\substack{x\in K, \, \lVert x \rVert \rightarrow+ \infty}} h(x) = + \infty.
 \end{equation}
 or equivalently, if $S_{\lambda} (h)=
 \{x\in K:\ h(x)\le \lambda\}$ is bounded for all \(\lambda \in \mathbb{R}\).
\end{itemize}

Clearly, $(i) \Rightarrow (ii)$, but the reverse statement does not hold as the function $h(x) = \lvert x \rvert$ shows. 

The following generalized convexity notion has been used in several applications from machine learning in virtue of its properties for fast convergence of  gradient-type methods (see \cite{GGK,HMR,HADR,HSS,WW} and references therein).

Motivated by \cite[Lemma 10]{HSS}, we define quasar-convexity as follows:
\begin{definition}\label{def:quasar}
Let $K \subset \mathbb{R}^{n}$ be a closed convex set, $\kappa \in \, ]0,1]$, $\gamma \geq 0$, and $h: \mathbb{R}^{n} \to \overline{\mathbb{R}}$ be a proper function such that $K \subset {\rm dom}\,h$. We say that $h$ is $(\kappa, \gamma)$-strongly quasar-convex on $K$ with respect to $\overline{x} \in {\rm argmin}_{K}\,h$ if
\begin{align}
 h(\lambda \overline{x}+(1-\lambda)x) \leq \kappa \lambda h(\overline{x})+(1-\kappa\lambda) h(x) - \lambda \left(1-\frac{\lambda}{2-\kappa}\right) \frac{\kappa \gamma}{2} \|x-\overline{x}\|^2,
\end{align}
for all $\lambda\in[0,1]$ and all $x\in K$.
Furthermore, when $\gamma=0$, we simply say that $h$ is $\kappa$-quasar-convex on $K$ with respect to $\overline{x}\in \operatorname{argmin}_{K} h$.
\end{definition}
Note that ${\rm argmin}_{K}\,h$ is a singleton when $h$ is strongly quasar-convex on $K$ ($\gamma > 0$), thus in this case we say that {\it $h$ is strongly quasar-convex with modulus $\kappa \in \, ]0, 1]$ and $\gamma > 0$}. If $h$ is quasar-convex ($\gamma = 0$), then ${\rm argmin}_{K}\,h$ is {\it star-shaped} (see \cite[Observation 3]{HSS}), i.e., there exists an $\overline{x}_{0} \in {\rm argmin}_{K}\,h$ such that for all $\overline{x} \in {\rm argmin}_{K}\,h$ and all $\lambda \in [0, 1]$, we have $\lambda \overline{x}_{0} + (1-\lambda)\overline{x} \in {\rm argmin}_{K}\,h$. Fur\-ther\-mo\-re, if $\kappa = 1$, then (strongly) quasar-convex functions are {\it (strongly) star-convex} \cite{NP-2006}. In particular, every (strongly) convex function is $(1, \gamma)$-strongly qua\-sar-convex ($\gamma \geq 0$).

The relationship between convex/star-convex/quasar-convex notions is su\-mma\-ri\-zed below:
 \begin{align}\label{scheme}
  \begin{array}{ccccccc}
  {\rm strongly ~ convex} & \overset{*}{\Longrightarrow} & {\rm strongly ~ star{\rm -}convex} & \overset{*}{\Longrightarrow} & {\rm strongly ~ quasar{\rm -}convex}  \notag \\
  \Downarrow & \, & \Downarrow & \, & \Downarrow  \notag \\
  {\rm convex} & \overset{*}{\Longrightarrow} & {\rm star{\rm -}convex} & \overset{*}{\Longrightarrow} & {\rm quasar{\rm -}convex},
  \end{array}
 \end{align}
 where ``$*$" denotes that ${\rm argmin}_{K\,}h \neq \emptyset$ is required. All the reverse statements do not hold in general (see \cite{GGK,HADR,HSS,NP-2006}).

\begin{remark}
 There is no relationship between (strongly) quasiconvex and (strong\-ly) quasar-convex functions. Indeed, the function in \cite[Example 9]{HLMV} is strong\-ly quasiconvex but not quasar-convex, while the function in Example \ref{exam:02} (below) is strongly quasar-convex but not quasiconvex. For differentiable func\-tions, sufficient conditions under which strong quasiconvexity implies quasar-convexity are provided in \cite{HLMV,LV}.
\end{remark}

Note that if $h$ is strongly quasar-convex on $K$ with modulus $\kappa \in \, ]0, 1]$ and $\gamma > 0$, then (see \cite[Corollary 1]{HSS})
\begin{equation}\label{qwc:sconvex}
 h(\overline{x}) + \frac{\kappa \gamma}{2(2 - \kappa)} \lVert y - \overline{x} \rVert^{2} \leq h(y), ~ \forall ~ y \in K,
\end{equation}
where $\overline{x} \in {\rm argmin}_{K}\,h$, thus strongly quasar-convex functions satisfy a quadratic growth condition with modulus $\frac{\kappa \gamma}{2(2 - \kappa)} > 0$. 

In the differentiable case, we have the following useful characterization: A differentiable function $h: \mathbb{R}^{n} \rightarrow \mathbb{R}$ is $(\kappa, \gamma)$-strongly quasar-convex ($\kappa \in \, ]0, 1]$ and $\gamma \geq 0$) with respect to $\overline{x} \in {\rm argmin}_{\mathbb{R}^{n}}\,h$ if and only if (see \cite[Lemma 10]{HSS})
\begin{equation}\label{diff:squasar}
 h(\overline{x}) \geq h(y) + \frac{1}{\kappa} \langle \nabla h(y), \overline{x} - y \rangle + \frac{\gamma}{2} \lVert y - \overline{x} \rVert^{2}, ~ \forall ~ y \in \mathbb{R}^{n}.
\end{equation}
 Hence, quasi-strongly convex functions \cite{NNG} are related to strongly quasar-convex.
 Further properties for differentiable (strongly) quasar-convex functions may be found in \cite{GGK,HMR,HADR,HSS,WW} among others.

Given a nonempty, closed and convex set $K \subset \mathbb{R}^{n}$, the proximity operator on $K$ of parameter $\beta > 0$ of a proper lsc  function $h: \mathbb{R}^{n} \rightarrow \overline{\mathbb{R}}$ at $z \in \mathbb{R}^{n}$ is defined as the operator ${\rm Prox}_{\beta h} (K, \cdot): \mathbb{R}^{n} \rightrightarrows \mathbb{R}^{n}$ for which
\begin{equation}\label{prox:operator}
 {\rm Prox}_{\beta h} (K, z) = \mathrm{argmin}_{x \in K} \left\{ h(x) + \frac{1}{2 \beta} \Vert z - x \rVert^{2} \right\}.
\end{equation}
When $K = \mathbb{R}^{n}$, we simply write ${\rm Prox}_{\beta h} (z) := {\rm Prox}_{\beta h} (\mathbb{R}^{n}, z)$. If $h$ is proper, lsc and convex, then ${\rm Prox}_{\beta h}$ turns out to be a single-valued ope\-ra\-tor (see, for instance, \cite[Proposition
12.15]{BC-2}).

For a further study on generalized convexity and proximal point type me\-thods we refer to \cite{ADSZ,BC-2,CMA,GL2,GL1,GLM-Survey,HKS,KL-2,Lara-9,LMV,P} and references therein.

\section{Nonsmooth Quasar-Convex Functions}\label{sec:03}

While quasar-convex functions are well-understood in the differentiable setting (see \cite{GGK,HMR,HADR,HSS,WW}), their nonsmooth counterparts remain largely unexplored.
This section addresses this gap by presenting key examples of nonsmooth quasar-convex functions and analyzing their properties, particularly their relationship with the proximity operator.

\subsection{Properties and Examples}\label{subsec:3-1}

 We begin with two auxiliary lemmas that will be used to construct explicit examples of homogeneous (strongly) quasar-convex functions. The first analyzes auxiliary functions while the second uses them to characterize quasar-convexity and strong quasar-convexity for homogeneous functions. These results will be used in the sequel.

\begin{lemma}\label{lemma:theta}
Let $K \subset \mathbb{R}^n$ be a nonempty closed convex set with $0 \in K$, and let $\alpha > 0$. Assume that $\kappa \in ~ ]0, \alpha[$ if $\alpha \leq 1$, and $\kappa \in ~ ]0,1]$ if $\alpha > 1$. Consider the func\-tions $\theta_\alpha: \, ]0,1] \to \mathbb{R}$ and $Q_{\alpha,\kappa}: \, ]0,1] \to \mathbb{R}$ defined by
 \begin{equation}\label{thetaQ}
  \theta_\alpha(\lambda) := \frac{1 - (1 - \lambda)^\alpha}{\lambda}, \quad Q_{\alpha,\kappa}(\lambda) := \frac{\theta_\alpha(\lambda) - \kappa}{\left(1 - \dfrac{\lambda}{2 - \kappa}\right)\dfrac{\kappa}{2}}.
 \end{equation}
 For $\kappa = 1$, we extend continuously $Q_{\alpha,1}(1)$ by taking $Q_{\alpha,1}(1) := 2$. Then,
\begin{itemize}
 \item[$(a)$] The function $\theta_\alpha$ is strictly increasing for $\alpha \in \, ]0,1[$, strictly decreasing for $\alpha > 1$, and identically equal to $1$ for $\alpha = 1$. Moreover,
 $$\inf_{\lambda \in\, ]0,1]} \theta_\alpha(\lambda) =
 \begin{cases}
  \alpha & {\rm if } ~  \alpha \in\, ]0,1], \\
  1 & {\rm if } ~ \alpha > 1. 
 \end{cases}
 $$

 \item[$(b)$] The function $Q_{\alpha,\kappa}(\lambda)$ is strictly increasing on $\, ]0,1]$ for $\alpha \in \, ]0,2[$, and strictly decreasing for $\alpha > 2$. Moreover,
 $$ \inf_{\lambda \in \, ]0,1]} Q_{\alpha,\kappa}(\lambda) =
 \begin{cases}
  \dfrac{2(\alpha - \kappa)}{\kappa} & {\rm if } ~ \alpha \in\, ]0,2[, \\
  \displaystyle\frac{2(2 - \kappa)}{\kappa} & {\rm if } ~ \alpha \geq 2.
 \end{cases}
 $$
 \end{itemize}
\end{lemma}

 The proof of Lemma \ref{lemma:theta} follows from a direct direct analysis of the functions $\theta_\alpha$ and $Q_{\alpha,\kappa}$, and is therefore omitted.

\begin{lemma}\label{lemmaC}
 Let $K \subset \mathbb{R}^n$ be a nonempty closed convex set with $0 \in K$, and $h : K \to \mathbb{R}$ be a function satisfying
 \begin{equation}\label{func:assump}
  h(tx) = t^\alpha h(x), ~ \forall ~ x \in K, ~ \forall ~ t \in [0,1],
 \end{equation}
 for some $\alpha > 0$, with $0 \in {\rm argmin}_K\,h$. Consider the functions $\theta_\alpha$ and $Q_{\alpha,\kappa}$ given in \eqref{thetaQ}. Then the following assertions hold:
 \begin{itemize}
  \item[$(a)$] The function $h$ is $\kappa$-quasar-convex on $K$ with respect to $\overline{x} = 0$ if and only if
  \begin{equation}\label{kappa:holds}
   0 < \kappa \leq \theta_\alpha (\lambda), ~ \forall ~ \lambda \in\, ]0,1].
  \end{equation}

  \item[$(b)$] The function $h$ is $(\kappa, \gamma)$-strongly quasar-convex on $K$ with respect to $\overline{x} = 0$ if and only if 
  \eqref{kappa:holds} holds and
  \begin{equation}\label{gamma:holds}
   0 \leq  \gamma \leq Q_{\alpha, \kappa}(\lambda) \frac{h(x)}{\lVert x \rVert^2}, ~ \forall ~ x \in K \setminus \{0\}, ~ \forall ~ \lambda \in \, ]0,1].
  \end{equation}
 \end{itemize}
\end{lemma}

\begin{proof}
 Since $\overline{x} = 0 \in {\rm argmin}_{K}\,h$ and $h$ satisfies \eqref{func:assump}, $h(0)=0$. Thus, $h$ is $(\kappa, \gamma)$-strong\-ly quasar-convex with respect to $\overline{x} = 0$ if and only if
\begin{align}
 & h((1 - \lambda)x) \leq (1 - \kappa \lambda) h(x) - \lambda \left(1 - \frac{\lambda}{2 - \kappa} \right) \frac{\kappa \gamma}{2} \lVert x\rVert ^2, \notag\\
 & \Longleftrightarrow (1-\lambda)^\alpha  \leq (1 - \kappa \lambda) - \lambda \left(1 - \frac{\lambda}{2 - \kappa} \right) \frac{\kappa \gamma}{2} \dfrac{\lVert x\rVert^2}{h(x)}. \label{hqc5} 
\end{align}

If $\gamma = 0$, then $h$ is $\kappa$-quasar-convex with respect to $\overline{x} = 0$ if and only if $0<\kappa \leq \theta_{\alpha}(\lambda)$, and $(a)$ follows. If $\gamma > 0$, then inequality \eqref{hqc5} is equivalent to
\begin{align*}
 & \theta_\alpha(\lambda) - \kappa \geq \left(1 - \frac{\lambda}{2 - \kappa} \right) \frac{\kappa \gamma}{2} \frac{\lVert x \rVert^2}{h(x)} ~ \Longrightarrow ~ \gamma \leq Q_{\alpha,\kappa}(\lambda) \frac{h(x)}{\lVert x \rVert^2}, 
\end{align*} 
for all $x \in K \setminus \{0\}$ and all $\lambda \in \, ]0,1]$, where $Q_{\alpha,\kappa}$ is defined in \eqref{thetaQ}. 
\end{proof}

As a consequence, we have the following.

\begin{proposition}\label{f.h}
 Let $K \subset \mathbb{R}^n$ be a nonempty closed convex set with $0 \in K$, and $h : K \to \mathbb{R}$ be a function satisfying \eqref{func:assump} with $0 \in {\rm argmin}_{K}\,h$. Then the following assertions hold:
\begin{itemize}
 \item[$(a)$] If $0 < \alpha < 1$, then $h$ is $\kappa$-quasar-convex on $K$ with respect to $\overline{x} = 0$ for every $\kappa \in \, ]0, \alpha]$.

 \item[$(b)$] If $\alpha \geq 1$, then $h$ is $\kappa$-quasar-convex on $K$ with respect to $\overline{x} = 0$ for every $\kappa \in \, ]0, 1]$.
\end{itemize}
\end{proposition}

\begin{proof}
 Direct consequence of Lemmas \ref{lemma:theta}$(a)$ and \ref{lemmaC}$(a)$, using that $\lim_{\lambda \downarrow 0} \theta_\alpha (\lambda) = \alpha$ and the monotonicity of $\theta_\alpha$.
\end{proof}

Sufficient conditions for ensuring strong quasar-convexity are given below.

\begin{proposition}\label{f.h1}
Let $K \subset \mathbb{R}^n$ be a nonempty compact convex set with $0 \in K$, and suppose there exists $c > 0$ such that
\begin{equation}\label{H}
 x \in K \backslash \{0\}, ~~ \lVert x \rVert < c ~ \Longrightarrow ~ \dfrac{c x}{\lVert x \rVert} \in K.
 \end{equation}
 \noindent Let $h : K \to \mathbb{R}$ be a lsc function satisfying \eqref{func:assump} with ${\rm argmin}_{K}\, h = \{0\}$. Define $M := \sup\{\lVert x \rVert : x \in K\} < + \infty$ and  $S_c := \min\{ h(x) : x \in K, \lVert x \rVert = c \}$. Then the following assertions hold:
 \begin{itemize}
  \item[$(a)$] If $0 < \alpha \leq 1$, then $h$ is $(\kappa, \gamma)$-strongly quasar-convex on $K$ with respect to $\overline{x} = 0$ for every
  $\kappa \in\, ]0, \alpha[$ and 
  $$\gamma := \frac{2(\alpha - \kappa)\, S_c}{\kappa\, c^\alpha\, M^{2 - \alpha}} > 0.$$

  \item[$(b)$] If $1 < \alpha < 2$, then $h$ is $(\kappa, \gamma)$-strongly quasar-convex on $K$ with respect to $\overline{x} = 0$ for every
  $\kappa \in \, ]0, 1]$ and 
  $$\gamma := \frac{2(\alpha - \kappa)\, S_c}{\kappa\, c^\alpha\, M^{2 - \alpha}} > 0.$$

  \item[$(c)$] If $\alpha = 2$, then $h$ is $(\kappa, \gamma)$-strongly quasar-convex on $K$ with respect to $\overline{x} = 0$ for every
  $\kappa \in\, ]0, 1]$ and 
  $$\gamma := \frac{2(2 - \kappa)\, S_c}{\kappa\, c^2} > 0.$$
 \end{itemize}
\end{proposition}

\begin{proof}
Since $K_c := \{x \in K : \lVert x \rVert = c\}$ is compact and $h$ is lsc, there exists $x^{0} \in K_c$ such that
$$
h(x^{0}) = \min_{x \in K_c} h(x) = S_c.
$$
As $x^{0} \neq 0$ and $\mathrm{argmin}_K\, h = \{0\}$, it follows that $S_c > 0$.
Let $x \in K \setminus \{0\}$ and take $v := \frac{c x}{\lVert x \rVert}$. By assumption \eqref{H}, we have $v \in K$ whenever $\lVert x\rVert  \leq c$, thus the convexity of $K$ ensures $v \in K$ when $\lVert x\rVert  > c$.
Then we have two cases:
\begin{itemize}
 \item[$(i)$] If $\lVert x\rVert  \leq c$, then $\frac{\lVert x\rVert }{c} \leq 1$, and using \eqref{func:assump} we have
 $$ h(x) = \left( \frac{\lVert x\rVert }{c} \right)^\alpha h(v) ~ \Longrightarrow ~ \frac{h(x)}{\lVert x\rVert ^2} = \frac{h(v)}{c^\alpha\, \lVert x\rVert ^{2 - \alpha}} \geq \frac{S_c}{c^\alpha M^{2 - \alpha}}.$$

 \item[$(ii)$] If $\lVert x \rVert > c$, then $\frac{c}{\lVert x \rVert}< 1$, and again by \eqref{func:assump} we have
 $$ h(v) = \left( \frac{c}{\lVert x \rVert} \right)^\alpha h(x) ~ \Longrightarrow ~ \frac{h(x)}{\lVert x \rVert^2} = \frac{h(v)}{c^\alpha} \cdot \lVert x \rVert^{\alpha - 2} \geq \frac{S_c}{c^\alpha M^{2 - \alpha}}.$$
\end{itemize}
Hence, in both cases,
$$\frac{h(x)}{\lVert x \rVert^2} \geq \frac{S_c}{c^\alpha M^{2 - \alpha}}, ~ \forall ~ x \in K \setminus \{0\}.
$$
Therefore, the conclusion follows from Lemmas \ref{lemma:theta}$(b)$ and \ref{lemmaC}$(b)$.
\end{proof}

\begin{remark}
Note that for $0<\alpha < 2$, the boundedness of $K$ is required. For $\alpha = 2$, this assumption is no longer necessary. However, when $\alpha > 2$, we have
$$\lim_{x \to 0} \ \frac{h(x)}{\lVert x \rVert^2} = \lim_{x \to 0} \left( h\left(\frac{x}{\lVert x \rVert} \right) \lVert x \rVert^{\alpha - 2} \right) = 0,$$
which shows that $h$ is not strongly quasar-convex in this case.
\end{remark}

The previous results have interesting consequences, illustrated by the fo\-llo\-wing examples of nonsmooth (strongly) quasar-convex functions.

Notably, the $\ell_p$-regularization ($0<p<1$) is quasar-convex. As far as we know, this is the first result describing the generalized convexity notion related to $\ell_p$-regularization functions.

\begin{example}\label{lp:func} 
 {\bf (The $\ell_{p}$-regularization with $0<p<1$)} Let $0<p<1$ and $h_{p}: \mathbb{R}^{n} \rightarrow \mathbb{R}$ be given by 
 \begin{equation}\label{lp:regu}
  h_{p} (x) = \lVert x \rVert_{p} ^p= \sum^{n}_{i=1} \lvert x_{i} \rvert^{p}.
  \end{equation} 
 Clearly, $h_{p}$ is not quasiconvex (see \cite[Remark 3.3]{AIKL}).
 
 It follows from Proposition \ref{f.h} that $h_{p}$ is quasar-convex on $\mathbb{R}^{n}$ with $\kappa \in \, ]0, p]$. Moreover, using the notation of Proposition \ref{f.h1}, and restricting $h_p$ to $K=\overline{\mathbb{B}}(0,1)$ with $c=1$, we obtain
$S_1=\min_{\|x\|=1}\|x\|_p^p=1.$
Therefore, Proposition \ref{f.h1} ensures that $h_p$ is strongly quasar-convex on $\overline{\mathbb{B}}(0,1)$ with

\begin{equation*}
 \kappa \in\, ]0, p[ ~~ {\rm and} ~~ \gamma_{p} = \dfrac{2(p - \kappa)}{\kappa} > 0.
\end{equation*}

 It is known that the $\ell_p$-regularization is an important tool in sparse optimization, signal/image recovery, LASSO and several applications in machine learning and statistics among others (see \cite{KF,Lu,Survey} and references therein).
\end{example}

In the following example, we show that the CES (Constant Elasticity of Substitution) production/utility function \cite{ACMS} is quasar-convex.

\begin{example}\label{ces:func}
 {\bf (The CES production/utility function)} Let $\beta \neq 0$ and $\alpha_{i} > 0$ for all $i \in \{1, \ldots, n\}$. Then the CES function is $h_{C}: \mathbb{R}^{n}_{+} \rightarrow \mathbb{R}$ given by (see \cite{ACMS,Mc})
 \begin{equation}\label{ces:function}
  h_{C} (x) = (\alpha_{1} x^{\beta}_{1} + \ldots + \alpha_{n} x^{\beta}_{n})^{\frac{1}{\beta}}.
 \end{equation}
 For $\beta < 0$, the function defined in \eqref{ces:function} can be continuously extended to $\mathbb{R}^n_{+}$ by setting $h_C(x) := 0$ for $x \in \mathbb{R}^n_{+} \setminus \mathbb{R}^n_{++}$.
 
 It is known that $h_{C}$ is quasiconcave if and only if $\beta \leq 1$, and it is convex if and only if  $\beta \geq 1$ (see \cite[Theorem 2.4.2]{CMA}).
 
 On the other hand, 
 it follows from Pro\-po\-si\-tion \ref{f.h} that $h_{C}$ is $\kappa$-quasar-convex with respect to $0 \in {\rm argmin}_{\mathbb{R}^{n}_{+}}\,h_{C}$ with modulus $\kappa \in \, ]0, 1]$ for all $\beta \neq 0$. Furthermore, for $\beta > 0$, $h_{C}$ is strongly quasar-convex on bounded convex sets in virtue of Proposition \ref{f.h1}. In particular, for $K = \overline{\mathbb{B}}(0,1) \cap \mathbb{R}^n_+$ and $c = 1$, we have: 
 \begin{itemize}
  \item[$(i)$] If $\beta \in\, ]0,2]$ and $x \in K$ with $\lVert x \rVert = 1$, then
  $$h_C(x) \geq \left( \alpha_{1} x^{2}_{1} + \ldots + \alpha_{n} x^{2}_{n} \right)^{\frac{1}{\beta}} \geq \min \{ \alpha_1, \dots, \alpha_n \}^{\frac{1}{\beta}}.$$
  Thus, $S_1 = \min \{ \alpha_1, \dots, \alpha_n \}^{\frac{1}{\beta}}$, and by Proposition \ref{f.h1}, we get
  $$\gamma_{C} = \frac{2(1 - \kappa)\, \min \{ \alpha_1, \dots, \alpha_n \}^{\frac{1}{\beta}}}{\kappa} > 0.$$

  \item[$(ii)$] If $\beta > 2$ and $x \in K$ with $\lVert x\rVert = 1$, by H\"{o}lder's inequality, we obtain
  \begin{align*}
   & 1 = \sum_{i=1}^n x_i^2 = \sum_{i=1}^n \alpha_i^{\frac{2}{\beta}} x_i^2 \alpha_i^{-\frac{2}{\beta}} \leq \left( \sum_{i=1}^n \alpha_i x_i^\beta \right)^{\frac{2}{\beta}} \left( \sum_{i=1}^n \alpha_i^{ -\frac{2}{\beta - 2} } \right)^{1 - \frac{2}{\beta}} \\
   & \hspace{1.7cm} \Longleftrightarrow ~ h_C(x) \geq \left( \sum_{i=1}^{n} \alpha_i^{ -\frac{2}{\beta - 2} } \right)^{ -\frac{\beta - 2}{2\beta}}.
   \end{align*}
  Furthermore, equality in H\"{o}lder's inequality occurs on the unit sphere for  $x_i = \frac{ \alpha_i^{-1/(\beta - 2)}}{\left( \sum_{j=1}^{n} \alpha_j^{-2/(\beta - 2)} \right)^{1/2}}$ for all $i \in \{1, \ldots, n\}$, which leads to
  $$\gamma_{C} = \frac{2(1 - \kappa)}{\kappa} \left( \sum_{i=1}^{n} \alpha_i^{ -\frac{2}{\beta - 2}}\right)^{ -\frac{\beta - 2}{2\beta}} > 0.$$
 \end{itemize}

 The importance of the CES function is well-known in economics in virtue of its applications in production and utility theory among others (see  \cite{ACMS,ADSZ,CD,D-1959,Leon,MWG,Mc} and references therein).
\end{example}

In particular, the Cobb-Douglas \cite{CD} and Leontief \cite{Leon} production/utility functions are quasar-convex, too. Since the Cobb-Douglas function is differentiable, we describe below the Leontief (perfect complements) function.

\begin{example}\label{Leon:func}
{\bf (The Leontief production/utility function)} Let $\alpha > 0$ and $\alpha_{i} > 0$ for all $i \in \{1, \ldots, n\}$. Then the Leontief (see  \cite{Leon}) production/utility function is $h_{L}: \mathbb{R}^{n}_{+} \rightarrow \mathbb{R}$ given by 
 \begin{equation}\label{Leontief}
  h_{L}(x) = \min\left\{ \frac{1}{\alpha_{1}} x_{1}, \ldots, \frac{1}{\alpha_{n}} x_{n} \right\}^{\alpha}.
 \end{equation}
 It is known that $h_{L}$ is quasiconcave, and it is concave if and only if $0 <\alpha \leq 1$ (see, for instance, \cite[Theorem 2.4.3]{CMA}). 
 
 On the other hand, since $h_{L} (tx) = t^{\alpha} h_{L} (x)$ and $0 \in {\rm argmin}_{ \mathbb{R}^{n}}\,h_{L}$, we have $h_{L}(0)=0$, thus the Leontief function is $\kappa$-quasar-convex with modulus $\kappa \in \, ]0, \alpha]$ when $0 < \alpha < 1$, and with modulus $\kappa \in \, ]0, 1]$ when $\alpha \geq 1$ by Proposition \ref{f.h}.  Fur\-ther\-more, $h_{L}$ is not necessarily strongly quasar-convex on bounded convex sets since it fails to have a unique minimizer at $0$ (every point with one zero coordinate is a minimizer of $h_{L}$). To be strongly quasar-convex, $h_{L}$ should be de\-fi\-ned on a bounded convex set $K$ of $\mathbb{R}^{n}_{+}$ with $0 \in K$  and $K$ should not contain any other point with a zero component. 

 The Leontief function is an important tool in economics in virtue of its a\-ppli\-cations in production and utility theory based on the relationship between different commodities/goods (see \cite{ADSZ,Leon,MWG,Mc} among others).
\end{example}

Also, the Euclidean norm with power $0 < \alpha \leq 2$ is (strongly) quasar-convex.

\begin{example}\label{norm:alpha}
 {\bf (Euclidean norm with power $0 < \alpha \leq 2$)} Let $0 < \alpha \leq 2$ and $h_{E}: \mathbb{R}^{n} \rightarrow \mathbb{R}$ be given by 
 \begin{equation}\label{eucli:norm}
  h_{E}(x) = \lVert x \rVert^{\alpha}.
 \end{equation} 
 It is known that $h_{E}$ is quasiconvex on $\mathbb{R}^{n}$ and strongly quasiconvex on bounded convex sets by \cite[Corollary 3.9]{NS} and \cite[Page 27]{VNC-2}.
 
 On the other hand, it follows from Proposition \ref{f.h} that $h_{E}$ is $\kappa$-quasar-convex with respect to $\overline{x} = 0 \in {\rm argmin}_{\mathbb{R}^{n}}\,h_{E}$ with modulus $\kappa \in \, ]0, \alpha]$ when $0< \alpha <1$, and with modulus $\kappa \in \, ]0, 1]$ when $\alpha \geq 1$; and in both cases, it is strongly quasar-convex on bounded convex sets in virtue of Proposition \ref{f.h1}.
 When $\alpha = 2$, $h_{E}$ is strongly convex, thus strongly quasar-convex (with $\kappa = 1$) by Proposition \ref{f.h1}$(c)$. 
In particular, for $K := \overline{\mathbb{B}} (0,1)$ and $c = 1$, it follows from Proposition \ref{f.h1} that
 $$ \gamma_{E} = \dfrac{2(\alpha - \kappa)}{\kappa} > 0,$$
 with $\kappa \in \, ]0, \alpha[$ when $\alpha \leq 1$, while $\kappa \in\, ]0, 1]$ when $\alpha \in\, ]1, 2]$.
\end{example}

\subsection{On the Proximity Operator}

The next result provides an analysis of the behaviour at infinity of strongly quasar-convex functions. 

\begin{proposition}\label{2super:quasar} 
 Let $K \subset \mathbb{R}^{n}$ be a closed convex set and $h: \mathbb{R}^{n} \rightarrow \overline{\mathbb{R}}$ be a pro\-per lsc function such that $K \subset {\rm dom}\,h$. If $h$ is  a strongly quasar-convex function on $K$ with modulus $\kappa \in \, ]0, 1]$ and $\gamma > 0$, then $h$ is $2$-supercoercive  on $K$.
\end{proposition}

\begin{proof}
 Let $\{x^{k}\}_{k} \subset K$ be such that $\lVert x^{k} \rVert \rightarrow + \infty$. Then,
 we may assume that $\frac{x^{k}}{\lVert x^{k} \rVert} \rightarrow u \in K^{\infty}$ with $\lVert u \rVert= 1$. Since $h$ has a minimum $\overline{x} \in K$, $h(\overline{x}) \leq h(x^{k})$ for all $k \in \mathbb{N}$. Using relation \eqref{qwc:sconvex}, we have
\begin{align*} 
 h(x^{k}) \geq h(\overline{x}) + \frac{\kappa \gamma}{2(2-\kappa)} \lVert x^{k} - \overline{x} \rVert^{2} & \Longrightarrow \frac{h(x^{k})}{\lVert x^{k} \rVert^{2}} \geq \frac{h(\overline{x})}{\lVert x^{k} \rVert^{2}} + \frac{\kappa \gamma}{2(2-\kappa)} \left \Vert \frac{x^{k}}{\lVert x^{k} \rVert} - \frac{\overline{x}}{\lVert x^{k} \rVert} \right \Vert^{2} \\
 & \Longrightarrow \liminf_{k \rightarrow+ \infty} \frac{h(x^{k})}{\lVert x^{k} \rVert^{2}} \geq \frac{\kappa \gamma}{2(2-\kappa)} > 0,
\end{align*}
 thus $h$ is 2-super\-coercive on $K$.
\end{proof}

As a consequence, we have the following result (which  includes also the case $\gamma = 0$). 

\begin{corollary}\label{exist:iter}
 Let $K \subset \mathbb{R}^{n}$ be a closed convex set and $h: \mathbb{R}^{n} \rightarrow \overline{\mathbb{R}}$ be a pro\-per, lsc and $(\kappa, \gamma)$-strongly quasar-convex function with respect to $\overline{x} \in {\rm argmin}_{K}\,h$ ($\kappa \in \, ]0, 1]$ and $\gamma \geq 0$) on $K \subset {\rm dom}\,h$. Then $h(\cdot) + \frac{1}{2 \beta} \lVert \cdot \rVert^{2}$ is lsc and $2$-super\-coer\-ci\-ve on $K$ for all $\beta> 0$. 
 As a consequence, ${\rm Prox}_{\beta h} (K, z)$ is nonempty and compact for all $z \in \mathbb{R}^{n}$ and all $\beta > 0$.
\end{corollary}

\begin{proof}
 Since $h$ has a minimum on $K$, the sum $h(\cdot) + \frac{1}{2 \beta} \lVert \cdot \rVert^{2}$ is lsc and $2$-super\-coer\-ci\-ve on $K$ for all $\beta> 0$. Hence, ${\rm Prox}_{\beta h} (K, z)$ is nonempty and compact for all $z \in \mathbb{R}^{n}$ and all $\beta > 0$ as a consequence of the Weierstrass Theorem.
\end{proof} 

\begin{remark}\label{not:singleton} 
 The proximity operator of a strongly quasar-convex function is not necessarily a singleton even in one dimension. Indeed, take $K=[-2, 2]$, $\beta = 1$, $z=\frac{3}{2}$ and $h: K \rightarrow \mathbb{R}$ given by $h(x) = \sqrt{\lvert x \rvert}$. Then, it follows from Proposition \ref{f.h1}$(a)$ that $h$ is strongly quasar-convex, but (see \cite[Remark 19$(ii)$]{Lara-9})
 $${\rm Prox}_{h}\left([-2,2], \tfrac{3}{2}\right) = \{0, 1\},$$
i.e., ${\rm Prox}_{h}$ is not necessarily a singleton.
\end{remark}


The following result will be useful in the sequel.

\begin{proposition}\label{sconvex}
 Let $K \subset \mathbb{R}^{n}$ be a closed convex set and $h: \mathbb{R}^{n} \rightarrow \overline{\mathbb{R}}$ be a pro\-per, lsc and $(\kappa, \gamma)$-strongly quasar-convex function with respect to $\overline{x} \in {\rm argmin}_{K}\,h$ ($\kappa \in \, ]0, 1]$ and $\gamma \geq 0$) on $K \subset {\rm dom}\,h$ and $z \in \mathbb{R}^{n}$. If $x^{*} \in {\rm Prox}_{\beta h} (K, z)$, then
 \begin{align}\label{lin:conv1}
  h(x^{*}) - h(\overline{x}) \leq \frac{1}{\kappa \beta} \langle x^{*} - z, \overline{x} - x^{*} \rangle - \frac{\gamma}{2} \lVert x^{*} - \overline{x} \rVert^{2}.
 \end{align}
\end{proposition}

\begin{proof}
 Since $x^{*} \in {\rm Prox}_{\beta h} (K, z)$, we have
 \begin{align}\label{ppa:dec}
  h(x^{*}) + \frac{1}{2 \beta} \lVert x^{*} - z \rVert^{2} & \leq h(x) + \frac{1}{2 \beta} \lVert x - z \rVert^{2}, ~ \forall ~ x \in K. 
\end{align} 
 Since $h$ is strongly quasar-convex with respect to $\overline{x} \in {\rm argmin}_{K}\,h$ ($\kappa \in \, ]0, 1]$ and $\gamma \geq 0$), taking $x = \lambda \overline{x} + (1 - \lambda) x^{*} \in K$ for all $\lambda \in [0, 1]$, we have
 \begin{align}
  & h(x^{*}) + \frac{1}{2 \beta} \lVert x^{*} - z \rVert^{2} \leq h(\lambda \overline{x} + (1 - \lambda) x^{*}) 
  + \frac{1}{2 \beta} \lVert \lambda 
  (\overline{x} - z) + (1-\lambda)(x^{*} - z) \rVert^{2}, \notag \\
  & \leq \lambda \kappa h(\overline{x}) + (1-\lambda \kappa) h(x^{*}) - \lambda \left(1 - \frac{\lambda}{2-\kappa} \right) \frac{\kappa \gamma}{2} 
  \lVert x^{*} - \overline{x} \rVert^{2} + \frac{\lambda}{2 \beta} \lVert z - \overline{x} \rVert^{2} \notag \\
  & ~~~ +  \frac{(1-\lambda)}{2 \beta} \lVert x^{*} - z \rVert^{2} - \frac{\lambda (1-\lambda)}{2 \beta} \lVert x^{*} - \overline{x} \rVert^{2}, ~ \forall ~ \lambda \in [0, 1], \notag
\end{align} 
in virtue of \eqref{iden:1}. Then, it follows by \eqref{3:points} that
 \begin{align*}
  & ~\, \lambda \kappa (h(x^{*}) - h(\overline{x})) \leq \frac{\lambda}{\beta} \langle x^{*} - z, \overline{x} - x^{*} \rangle + \frac{\lambda \kappa}{2} \left( \frac{\lambda}{\kappa \beta} - \gamma + \frac{\lambda \gamma}{2-\kappa} \right) \lVert x^{*} - \overline{x} \rVert^{2}, 
 \end{align*}
 and the result follows by multiplying by $\frac{1}{\lambda \kappa}$, and then taking $\lambda \downarrow 0$.
\end{proof}

\begin{remark} 
\begin{itemize}
 \item[$(i)$] In contrast to the convex and (strongly) quasiconvex cases (see \cite[Proposition 16.44]{BC-2} and \cite[Proposition 7]{Lara-9}, respectively), relation \eqref{lin:conv1} holds only for $\overline{x} \in {\rm argmin}_{K}\,h$. Surprisingly, as we will see in the next section, the information provided by \eqref{lin:conv1} is enough to guarantee (linear) convergence to the optimal solution of a (strongly) quasar-convex function and to provide a detailed complexity analysis.
 
 \item[$(ii)$] Recall that if $K \subset \mathbb{R}^{n}$ is a closed set and $h: \mathbb{R}^{n} \rightarrow \overline{\mathbb{R}}$ is a proper function with $K \cap {\rm dom}\,h \neq \emptyset$, then $h$ is said to be prox-convex on $K$ (see \cite{GL2,GL1}) if there exists $\alpha > 0$ such that for every $z \in K$, ${\rm Prox}_{h} (K, z) \neq \emptyset$, and
 \begin{equation}\label{prox:all}
  x^{*} \in {\rm Prox}_{h} (K, z) ~ \Longrightarrow ~ h(x^{*}) - h(x) \leq \alpha \langle x^{*} - z, x - x^{*} \rangle, ~ \forall ~ x \in K.
 \end{equation}
 It follows from Proposition \ref{sconvex} that prox-convex functions with modulus $\alpha \geq 1$ are quasar-convex with modulus $\kappa = \frac{1}{\alpha}$ and $\gamma = 0$.

 \item[$(iii)$] Observe that the function in Example \ref{lp:func} is strongly quasar-convex on boun\-ded convex sets while being neither weakly convex, nor DC, nor strongly quasiconvex, i.e., the results from \cite{AN,ABS,BB,CP,IPS,Lara-9,PEN,SSC} and re\-fe\-ren\-ces therein cannot be applied to this function.
 \end{itemize}
\end{remark}

The following result provides a connection between fixed points of the pro\-xi\-mi\-ty operator and minimal points of quasar-convex functions. We emphasize that function $h$ does not need to be quasar-convex with respect to every point in ${\rm argmin}_{K}\,h$, only with respect to one point $\overline{x} \in {\rm argmin}_{K}\,h$ is enough.

\begin{proposition}\label{prop:fixed}
 Let $K \subset \mathbb{R}^{n}$ be a closed convex set and $h: \mathbb{R}^{n} \rightarrow \overline{\mathbb{R}}$ be a pro\-per, lsc and $(\kappa, \gamma)$-strongly quasar-convex function with respect to $ \overline{x}\in{\rm argmin}_{K}\,h$ ($\kappa \in \, ]0, 1]$ and $\gamma \geq 0$) on $K \subset {\rm dom}\,h$, and $\beta>0$. Then
 \begin{equation}\label{fixed:points}
 {\rm Fix} \left( {\rm Prox}_{\beta h} (K, \cdot) \right) = {\rm argmin}_{K}\,h.
\end{equation}
\end{proposition}

\begin{proof}
 $(\supset)$ Straightforward. 
 
 $(\subset)$ Let $x^{*} \in {\rm Prox}_{\beta h} (K, x^{*})$. Then, $h(x^{*}) \leq h(x) + \frac{1}{2 \beta} \lVert x^{*} - x \rVert^{2}$ for all $x \in K$. Take $x = \lambda \overline{x} + (1-\lambda) x^{*} \in K$ with $\lambda \in [0, 1]$ and $\overline{x} \in {\rm argmin}_{K}\,h$. Since $h$ is $(\kappa,\gamma)$-strongly quasar-convex on $K$ with respect to $\overline{x}$, we have
\begin{align*}
 & h(x^{*}) \leq h(\lambda \overline{x} + (1-\lambda) x^{*}) + \frac{1}{2 \beta} \lVert \lambda (\overline{x} - x^{*} ) \rVert^{2} \\
 & \hspace{0.85cm} \leq \lambda \kappa h(\overline{x}) + (1-\lambda \kappa) h(x^{*}) - \lambda \left( 1 - \frac{\lambda}{2-\kappa} \right) \frac{\kappa \gamma}{2} \lVert x^{*} - \overline{x} \rVert^{2} + \frac{\lambda^{2}}{2 \beta} \lVert x^{*} - \overline{x} \rVert^{2}. 
\end{align*}
Then
 \begin{align*} 
  & h(x^{*}) - h(\overline{x}) \leq \frac{1}{2} \left( \frac{\lambda}{\kappa \beta} - \gamma + \frac{\lambda \gamma}{2-\kappa} \right) \lVert x^{*} - \overline{x} \rVert^{2}, ~ \forall ~ \lambda \in \, ]0, 1]. 
\end{align*}
Taking $\lambda \downarrow 0$, we obtain $h(x^{*}) \leq h(\overline{x})$, thus $x^{*} \in {\rm argmin}_{K}\,h$. 
\end{proof}

The previous result extends \cite[Proposition 12.29]{BC-2} from the convex to the quasar-convex case.

\section{The Proximal Point Algorithm}\label{sec:04}

Motivated by the lack of algorithms for nonsmooth quasar-convex functions, we develop a proximal point algorithm that builds directly on the foundational properties of proximity operators established in the previous section.

The Proximal Point Algorithm (PPA) that we consider is based on the cla\-ssi\-cal version of Rockafellar \cite{rock-1976} (see also \cite{M1,M2}).

Throughout this section, let $K \subset \mathbb{R}^{n}$ be a nonempty closed convex set and $h:\mathbb{R}^{n}\to\overline{\mathbb{R}}$ be a proper and lsc function such that $K \subset {\rm dom}\,h$.

\begin{algorithm}[H]
\caption{PPA for Quasar-Convex Functions (PPA-Quasar)}\label{ppa:sqcx}
\begin{description}
 \item[Step 0.] Choose $x^{0} \in K$, set $k=0$, and let $\{\beta_{k}\}_{k}$ be a sequence of positive numbers.

 \item[Step 1.] Choose
 \begin{equation}\label{step:sqcx}
  x^{k+1} \in {\rm Prox}_{\beta_{k} h} (K, x^{k}).
 \end{equation}

 \item[Step 2.] If $x^{k+1} = x^{k}$, then STOP, $x^{k} \in {\rm argmin}_{K}\,h$. Otherwise, set $k=k+1$ and go to Step 1.
 \end{description}
\end{algorithm}

Note that, under the quasar-convexity assumptions considered in this paper, the iterative steps are well defined by Corollary \ref{exist:iter}, while the stopping criterion follows from Proposition \ref{prop:fixed}.

\subsection{Convergence Analysis}

We proceed directly with the convergence result.

\begin{theorem}\label{usual:prox}
Let $K \subset \mathbb{R}^{n}$ be a closed convex set, and let $h: \mathbb{R}^{n} \rightarrow \overline{\mathbb{R}}$ be a proper and lsc function such that $K \subset {\rm dom}\,h$. Assume that $h$ is strongly quasar-convex on $K$ with modulus $\kappa \in \, ]0, 1]$ and $\gamma > 0$. Suppose that there exists $\beta^{\prime} > 0$ such that
\begin{equation}\label{param:cond1}
 0 < \beta^{\prime} \leq \beta_{k}, ~ \forall ~ k \in \mathbb{N}_{0}.
\end{equation}
Then the sequence $\{x^{k}\}_{k}$, generated by Algorithm \ref{ppa:sqcx}, is a minimizing sequence of $h$ on $K$, i.e., $h(x^{k}) \downarrow \min_{x \in K}\,h(x)$ and, moreover, converges linearly to the unique solution $\overline{x} \in {\rm argmin}_{K}\,h$ with convergence rate of at least 
\begin{equation}\label{conv:rate}
\frac{1}{\sqrt{1 + \kappa \beta^{\prime} \gamma + (\frac{\kappa^{2} \beta^{\prime} \gamma}{2 - \kappa})}} \, \in ~ ]0, 1[. 
\end{equation}
\end{theorem}

\begin{proof}
If $x^{k+1} = x^{k}$ for some $k \in \mathbb{N}_{0}$, then the algorithm stops and $x^{k} \in {\rm argmin}_{K}\,h$ by Proposition \ref{prop:fixed}. So, we assume that $x^{k+1} \neq x^{k}$ for all $k \in \mathbb{N}_{0}$.

Since $x^{k+1} \in {\rm Prox}_{\beta_{k} h} 
(K, x^{k})$, taking $z=x^{k}$, $x^{*}=x^{k+1}$ and $x=x^{k}$ in \eqref{ppa:dec}, and since we are assuming $x^{k+1} \neq x^{k}$, we obtain
\begin{align}\label{tel:one}
 h(x^{k+1}) < h(x^{k+1}) + \frac{1}{2 \beta_{k}} \lVert x^{k+1} - x^{k} \Vert^{2} \leq h(x^{k}).
\end{align}
Hence, $\{h(x^{k})\}_{k}$ is a decreasing sequence.

Now, let us prove that the sequence $\{x^{k}\}_{k}$ converges to the unique solution. Indeed, since $h$ is strongly quasar-convex and $x^{k+1} \in {\rm Prox}_{\beta_{k} h} (K,x^{k})$, it follows from  Proposition \ref{sconvex} with $\overline{x} \in {\rm argmin}_{K}\,h$ that
 \begin{align} \label{alpha:subd}
  & 2 \langle x^{k+1} - x^{k}, x^{k+1} - \overline{x} \rangle \leq 2 \kappa \beta_{k} (h(\overline{x}) - h(x^{k+1})) - \kappa \beta_{k} \gamma \lVert x^{k+1} - \overline{x} \rVert^{2}.
 \end{align}
 On the other hand, 
 \begin{align}
  \lVert x^{k+1} - \overline{x} \rVert^{2} & = \lVert x^{k+1} - x^{k} + x^{k}
  - \overline{x} \rVert^{2} \notag \\
  & = \lVert x^{k+1} - x^{k} \rVert^{2} + \lVert x^{k} - \overline{x} \rVert^{2}
  + 2\langle x^{k+1} - x^{k}, x^{k} - \overline{x} \rangle \notag \\
  & = \lVert x^{k} - \overline{x} \rVert^{2} -\lVert x^{k+1} - x^{k} \rVert^{2} + 2 \langle x^{k+1} - x^{k}, x^{k+1} - \overline{x} \rangle. \notag 
 \end{align}
Then, using \eqref{alpha:subd}, 
\begin{equation}\label{lin:conv}
 (1 + \kappa \beta_{k} \gamma) \lVert x^{k+1} - \overline{x} \rVert^{2} \leq \lVert x^{k} - \overline{x} \rVert^{2} -\lVert x^{k+1} - x^{k} \rVert^{2} + 2 \kappa \beta_{k} (h(\overline{x}) - h(x^{k+1})).
\end{equation}
In particular,
\begin{equation}\label{for:values}
 2 \kappa \beta' (h(x^{k+1}) - h(\overline{x})) \leq \lVert x^k - \overline{x} \rVert^2-\Vert x^{k+1}-\overline{x}\Vert^2.
\end{equation} 
Summing this inequality over $k$,
$$
\sum_{k=0}^{\ell} 2 \kappa \beta' (h(x^{k+1}) - h(\overline{x})) \leq \lVert x^{0} - \overline{x} \rVert^{2}-\lVert x^{\ell + 1} - \overline{x} \rVert^{2}, ~ \forall ~ \ell \in \mathbb{N}.
$$
Then the series $\sum_{k=0}^{\infty} (h(x^k)-h(\overline{x}))$ converges and, in particular, $\lim_{k\to \infty} h(x^k) = h(\overline{x}) = \min_{x \in K} h(x)$, i.e, $\{x^{k}\}_{k}$ is a minimizing sequence of $h$ on $K$.

\noindent Using now \eqref{qwc:sconvex} (with $y=x^{k+1}$) in \eqref{lin:conv}, we have
\begin{align}
 & (1 + \kappa \beta_{k} \gamma) \lVert x^{k+1} - \overline{x} \rVert^{2} \leq \lVert x^{k} - \overline{x} \rVert^{2} -\lVert x^{k+1} - x^{k} \rVert^{2} -  \frac{\kappa^{2} \beta_{k} \gamma}{2 - \kappa} \lVert x^{k+1} - \overline{x} \rVert^{2} \notag \\
 & \Longleftrightarrow \, \left( 1 + \kappa \beta_{k} \gamma + \frac{ \kappa^{2} \beta_{k} \gamma}{2 - \kappa} \right) \lVert x^{k+1} - \overline{x} \rVert^{2} \leq \lVert x^{k} - \overline{x} \rVert^{2} -\lVert x^{k+1} - x^{k} \rVert^{2} \notag \\
 & \, \overset{\eqref{param:cond1}}{\Longrightarrow} \, \lVert x^{k+1} - \overline{x} \rVert^{2} \leq \frac{1}{1 + \kappa \beta^{\prime} \gamma + (\frac{\kappa^{2} \beta^{\prime} \gamma}{2 - \kappa})} \lVert x^{k} - \overline{x} \rVert^{2}.
 \label{conv:rate1}
\end{align}
Hence, $\{x^{k}\}_{k}$ converges to the unique solution $\overline{x} \in {\rm argmin}_{K}\,h$, and using \eqref{tel:one} we conclude that it is a minimizing sequence, that is, $h(x^{k}) \downarrow \min_{x \in K}\,h(x)$. 

Finally, it follows from \eqref{conv:rate1} that $\{x^{k}\}_{k}$ converges linearly to $\overline{x} \in {\rm argmin}_{K}\,h$ with a convergence rate of at least $\frac{1}{\sqrt{1 + \kappa \beta^{\prime} \gamma + (\frac{ \kappa^{2} \beta^{\prime} \gamma}{2 - \kappa})}} \in \, ]0, 1[$. 
\end{proof}

If in particular the function $h$ is strongly star-convex on $K$ ($\kappa = 1$), we obtain:

\begin{corollary}\label{coro:starconvex}
 Under the assumptions of Theorem \ref{usual:prox}, if the function $h$ is strong\-ly star-convex ($\kappa = 1$, $\gamma > 0$), then the sequence $\{x^{k}\}_{k}$, generated by Algorithm \ref{ppa:sqcx}, is a minimizing sequence of $h$ and converges linearly to the unique solution $\overline{x} \in {\rm argmin}_{K}\,h$ with a convergence rate of at least 
\begin{equation}\label{conv:rate2}
  \frac{1}{\sqrt{1 + 2 \beta^{\prime} \gamma}} \, \in ~ ]0, 1[.
\end{equation}
\end{corollary}

\begin{remark}\label{rem:compa1}
 \begin{itemize}
  \item[$(i)$] Observe that in relation \eqref{param:cond1}, the parameter $\beta^{\prime} > 0$ can be chosen arbitrarily large, just as in the strongly convex case. This suggests that the convergence rate in \eqref{conv:rate} and \eqref{conv:rate2} could be made arbitrarily fast. However, if $\beta^{\prime} > 0$ is chosen too large, solving the subproblem \eqref{step:sqcx} becomes as expensive as solving the original problem $\min_{x \in K}\,h(x)$.

  \item[$(ii)$] The linear convergence result in Theorem \ref{usual:prox} is a proper extension since when $\kappa = 1$ in \eqref{conv:rate}, the linear rate is exactly as good as the rate for strongly convex minimization (see, for instance, \cite[Example 23.40$(ii)$]{BC-2}). 
 \end{itemize}
\end{remark}

In the case when $\gamma=0$, i.e., $h$ is quasar-convex, we can also ensure that the generated sequence is a minimizing sequence of $h$ on $K$.

\begin{theorem}\label{case gamma=0}
Let $K \subset \mathbb{R}^{n}$ be a closed convex set, and let $h: \mathbb{R}^{n} \rightarrow \overline{\mathbb{R}}$ be a proper and lsc function such that $K \subset {\rm dom}\,h$. Suppose that $h$ is $\kappa$-quasar-convex on $K$ with respect to every point of ${\rm argmin}_{K}\,h$, for some $\kappa \in \, ]0,1]$, and that condition \eqref{param:cond1} holds. Then the se\-quen\-ce $\{x^{k}\}_{k}$, generated by Algorithm \ref{ppa:sqcx}, is a minimizing sequence of $h$ on $K$, i.e., $h(x^{k}) \downarrow \min_{x \in K}\,h(x)$.
\end{theorem}
\begin{proof}
 Since $h$ is quasar-convex ($\gamma = 0$) and $x^{k+1} \in {\rm Prox}_{\beta_{k} h} (K,x^{k})$, it follows from \eqref{tel:one} that $\{h(x^{k})\}_{k}$ is decreasing. Also, by repeating the proof of Theorem \ref{usual:prox} with $\overline{x} \in {\rm argmin}_{K}\,h$ and $\gamma = 0$, it follows from \eqref{lin:conv} that
 \begin{align}\label{fejer}
  \lVert x^{k+1} - \overline{x} \rVert^{2} & \leq \lVert x^{k} - \overline{x} \rVert^{2}  +  2\kappa \beta_{k} (h(\overline{x}) - h(x^{k+1})) \notag \\
  & \leq \lVert x^{k} - \overline{x} \rVert^{2}  +  2 \kappa \beta^{\prime} (h(\overline{x}) - h(x^{k+1})).
 \end{align} 
Hence, the sequence $\{x^k\}_{k}$ is Fej\'er monotone with res\-pect to ${\rm argmin}_{K}\,h$ (see \cite[Definition 5.1]{BC-2}) and $2 \kappa \beta' \sum_{k=0}^{\infty} (h(x^{k+1})-h(\overline{x}))$ converges (by the same argument than in \eqref{for:values}). In particular, $\lim_{k\to \infty} h(x^k) = h(\overline{x}) = \min_{x \in K} h(x)$.

Since the sequence $\{\lVert x^k-\overline{x}\rVert^2\}_{k}$ is monotonically decreasing and thus boun\-ded, the sequence $\{x^k\}_{k}$ possesses a convergent subsequence $\{x^{j_k}\}_{k}$ with limit point $\hat{x}$. By the lower semicontinuity of $h$, we obtain
\begin{equation}
h(\hat{x}) \leq \liminf_{k\to\infty} h(x^{j_k}) = \lim_{k\to\infty} h(x^k) = \min_{x\in K} h(x).
\end{equation}
Consequently, $\hat{x}\in {\rm argmin}_{K}\,h$.

Thus, the sequence $\{x^{k}\}_{k}$ is Fejér monotone with respect to $ {\rm argmin}_{K}\,h$ and has at least one cluster point belonging to ${\rm argmin}_{K}\,h$. Applying \cite[Theorem 5.5]{BC-2}, we conclude that $\{x^{k}\}_{k}$ converges to some point in ${\rm argmin}_{K}\,h$.
\end{proof}

\begin{remark}
 \begin{itemize}
  \item[$(i)$] In the proof of the Theorem \ref{case gamma=0}, the conclusion that $\{x^k\}_{k}$ is a minimizing sequence, i.e., $h(x^k) \to \min_K h$, actually requires the quasar-convexity condition with respect to a single point $\overline{x} \in {\rm argmin}_K\,h$. The stronger assumption that the function is quasar-convex with respect to every point $\overline{x} \in {\rm argmin}_K\,h$ is needed in the subsequent argument for establishing convergence of $\{x^k\}_{k}$ to a minimizer.

  \item[$(ii)$] We note that there exists functions which are quasar-convex with respect to some minimizers and they are not quasar-convex with respect to other minimizers. Indeed, consider the function $h:\mathbb{R}^2\to\mathbb{R}$ given by $h(x,y) = x^2(1+y^2)$. Then ${\rm argmin}_{ \mathbb{R}^2}\,h =\{(0,t): t \in \mathbb{R}\}$ and $\min h=0$. 
  
  Observe that for $(\overline{x}, \overline{y}) =(0,t) \in{\rm argmin}_{\mathbb{R}^2}\,h$ and $\kappa\in(0,1]$, the function $h$ is $\kappa$-quasar-convex with respect to $(\overline{x}, \overline{y}) = (0,t)$ if and only if the inequality (see \eqref{diff:squasar}) 
  $$\langle \nabla h(x,y), (x,y)-(0,t) \rangle \ge \kappa h(x,y),$$
   holds for every $(x,y) \in \mathbb{R}^2$, that is, $2x^2(1+2y^2-ty) \ge \kappa x^2(1+y^2)$. 
   For $x\neq0$, we obtain 
   $$(4-\kappa) y^2 - 2ty + (2-\kappa) \ge0, ~ \forall ~ y \in \mathbb{R},$$ 
   which holds if and only if $|t| \le \sqrt{(4-\kappa) (2-\kappa)}$. Therefore, $h$ is $\kappa$-quasar-convex with respect to $(\overline{x}, \overline{y}) = (0,t)$ if $|t| \le \sqrt{(4-\kappa) (2-\kappa)}$, but for $|t| \ge 2 \sqrt2$, $h$ is not $\kappa$-quasar-convex with respect to $(\overline{x}, \overline{y}) = (0, t)$ for any choice of $\kappa \in \, ]0, 1]$. 
 \end{itemize} 
\end{remark}

In particular, for the star-convex case we have:

\begin{corollary}\label{starconvex}
 Under the assumptions of Theorem \ref{case gamma=0}. If the function $h$ is star-convex ($\kappa = 1$, $\gamma = 0$), then the sequence $\{x^{k}\}_{k}$, generated by Algorithm \ref{ppa:sqcx}, is a minimizing sequence of $h$. 
\end{corollary}

\subsection{Complexity Analysis}

In the following result, we present non-asymptotic estimates for both functional values $h(x^{k}) - \min_{x\in K} h(x)$ and the distance $\lVert x^{k} - \overline{x} \rVert$ when the objective function is strongly quasar-convex with $\gamma > 0$.

\begin{proposition}\label{prop:image}
 Suppose that the assumptions of Theorem \ref{usual:prox} hold and $\overline{x} \in {\rm argmin}_{K}\,h$. Let $\{x^k\}_{k}$ be the sequence generated by Algorithm \ref{ppa:sqcx}. Then, given a to\-le\-ran\-ce $\varepsilon > 0$, the following assertions hold:
 \begin{itemize}
 \item[$(a)$] At most  
\begin{equation*}
  \mathcal{O}( \ln( \varepsilon^{-1} )) = \dfrac{ \ln(\varepsilon^{-1}) + \ln(\lVert x^{0} - \overline{x} \rVert)}{\ln \left( \sqrt{1 + \kappa \beta^{\prime} \gamma + \frac{\kappa^{2} \beta^{\prime} \gamma}{2 - \kappa} } \right)}. 
 \end{equation*}
 iterations are needed to satisfy $\lVert x^{k}-\overline{x}\rVert \leq \varepsilon$.

 \item[$(b)$] At most  
 \begin{equation*}
   \mathcal{O} (\ln( \varepsilon^{-1} )) = \dfrac{\ln \left(\varepsilon^{-1} \right) + \ln \left( \dfrac{\lVert x^{0} - \overline{x} \rVert ^2}{{2 \kappa \beta^{\prime}}} \right)}{\ln \left(1 + \kappa \beta^{\prime} \gamma + \frac{\kappa^{2} \beta^{\prime} \gamma}{2 - \kappa} \right)} + 1. 
 \end{equation*}
 iterations are needed to satisfy $h(x^{k})- h(\overline{x}) \leq \varepsilon$.
 \end{itemize}
\end{proposition}

\begin{proof} 
$(a)$ It follows from Theorem \ref{usual:prox} that 
 \begin{equation}\label{conve:linear}
\lVert x^k-\overline{x}\rVert \leq \mu^{k} \lVert x^{0} - \overline{x} \rVert,
\qquad
\mu := \frac{1}{ \sqrt{1 + \kappa \beta^{\prime} \gamma + \frac{ \kappa^{2} \beta^{\prime} \gamma}{2 - \kappa}} } \in\,]0,1[.
\end{equation}
Therefore, $\lVert x^{k} - \overline{x} \rVert \leq \varepsilon$ is guaranteed whenever $
\mu^k \lVert x^{0} - \overline{x} \rVert \leq \varepsilon$,
that is,
\[
k \geq \dfrac{\ln(\varepsilon^{-1}) + \ln(\lVert x^{0} - \overline{x} \rVert)}{\ln \left(\sqrt{1 + \kappa \beta^{\prime} \gamma + \frac{\kappa^{2} \beta^{\prime} \gamma}{2 - \kappa} }\right)}.
\] 
Thus,  $\lVert x^{k} - \overline{x} \rVert \leq \varepsilon$ is guaranteed within $\mathcal{O}( \ln( \varepsilon^{-1} ))$ iterations..

\vspace{0.5cm}
\noindent $(b)$ From \eqref{lin:conv} and \eqref{conve:linear}, it follows that, for every $k \geq 1$,
\begin{equation*}
  h(x^{k}) - h(\overline{x})\leq \dfrac{1}{2\kappa \beta_{k-1}}\lVert x^{k-1}-\overline{x}\rVert ^2 \leq \dfrac{1}{2\kappa \beta^{\prime}}\mu ^{2(k-1)}\lVert x^{0}-\overline{x}\rVert ^2
 \end{equation*}
 and the result follows by solving $\left(\frac{1}{1 + \kappa \beta^{\prime} \gamma + \frac{\kappa^{2} \beta^{\prime} \gamma}{2 - \kappa}}  \right)^{k-1} \dfrac{1}{2 \kappa \beta^{\prime}} \lVert x^{0} - \overline{x} \lVert^2 \leq \varepsilon$. 
 
\end{proof}

\begin{corollary}\label{coro:starconvex1}
 Under the assumptions of Proposition \ref{prop:image}. If in particular the function $h$ is strongly star-convex ($\kappa = 1$, $\gamma > 0$), then: 
\begin{itemize}
  \item[$(a)$] At most $\mathcal{O}( \ln( \varepsilon^{-1})) = \dfrac{ \ln(\varepsilon^{-1}) + \ln(\lVert x^{0} - \overline{x} \rVert)}{\ln \left(1 + 2 \beta^{\prime} \gamma \right)}$, iterations are needed to satisfy $\lVert x^{k}-\overline{x}\rVert \leq \varepsilon$.

 \item[$(b)$] At most $\mathcal{O}( \ln( \varepsilon^{-1})) = \dfrac{\ln \left(\varepsilon^{-1} \right) + \ln\left( \dfrac{\lVert x^{k} - \overline{x} \lVert^2}{{2  \beta^{\prime}}}\right)}{\ln \left(1 + 2\beta^{\prime} \gamma \right)}+1$, iterations are needed to satisfy $h(x^{k})- h(\overline{x}) \leq \varepsilon$.
 \end{itemize}
\end{corollary}

Now, we analyze the general quasar-convex case ($\gamma = 0$).

\begin{proposition}\label{prop:reqexe} Under the assumptions of Theorem \ref{case gamma=0}, let $\overline{x} \in {\rm argmin}_{K}\,h$, and suppose that the sequence $\{x^k\}_{k}$ is generated by Algorithm~\ref{ppa:sqcx}, with step sizes satisfying
\begin{equation}\label{param:cond2}
 0 < \beta^{\prime} \leq \beta_{k} \leq \beta^{\prime \prime}, ~ \forall ~ k \in \mathbb{N}_0.
\end{equation}
Then, given a tolerance $\varepsilon > 0$, the following assertions hold:
\begin{itemize}
 \item[$(a)$] At most  
 \begin{equation}\label{comp:bound}
 \frac{2 \beta^{\prime \prime} \big(h(x^0) - \min_{x \in K} h(x)\big)}{\varepsilon^{2}},
 \end{equation}
 iterations are needed to satisfy the stopping criterion $\Vert x^{k+1} - x^k \Vert \leq \varepsilon$.

 \item[$(b)$] At most  
 \begin{equation}\label{comp:bound1}
 \frac{\lVert x^0 - \overline{x}\rVert ^2}{ 2\kappa \beta^{\prime} \varepsilon},
 \end{equation}
 iterations are needed to satisfy the stopping criterion $h(x^k)-h(\overline{x}) \leq \varepsilon$.
\end{itemize}
\end{proposition}
\begin{proof}
 $(a)$ The proof is standard (see, for instance, \cite[Proposition 4.8]{AIKL}). 

\noindent $(b)$ Using \eqref{lin:conv}, we have
\begin{align}
 & \hspace{1.65cm} 2 \kappa\beta^{\prime} (h(x^{k+1}) - h(\overline{x})) \leq \lVert  x^k - \overline{x} \rVert^2 - \lVert x^{k+1} - \overline{x}\rVert ^2, ~ \forall ~ k \geq 0 \notag \\
 & \Longrightarrow ~ 2\kappa \beta^{\prime}  \sum_{k=0}^{N-1} \left( h(x^{k+1}) - h(\overline{x}) \right) \leq \lVert x^0 - \overline{x}\rVert ^2 - \lVert x^N - \overline{x}\rVert ^2 \leq \lVert x^0 - \overline{x}\rVert ^2. \label{sum_bound}
\end{align}
Since $\{h(x^k)\}_{k}$ is nonincreasing, 
$h(x^N) - h(\overline{x}) \leq \frac{1}{N} \sum_{k=0}^{N-1} \left( h(x^{k+1}) - h(\overline{x}) \right)$.
Combining this with \eqref{sum_bound}, we conclude
\begin{equation} \label{bound_3}
 h(x^N) - h(\overline{x}) \leq \frac{\lVert x^0 - \overline{x}\rVert ^2}{2\kappa \beta^{\prime} N},
\end{equation}
and since this holds for all $N \in \mathbb{N}$, the result \eqref{comp:bound1} follows.
\end{proof} 

\begin{remark}
 Note that our complexity rates for (strongly) quasar-convex function preserve the same order as for (strongly) convex functions (see \cite[Theorem 2.1]{Gu} and \cite[Example 23.40]{BC-2} for instance), even when several classes of nonconvex functions are including in our analysis. 
\end{remark}

\section{Numerical Experiments}\label{sec:05}

In this section, we illustrate that, as in the convex case, solving a sequence of proximal subproblems may improve the conditioning of the problem and thus its numerical behavior.

In the examples below, we consider classes of functions which are neither convex nor quasiconvex, since their sublevel sets are not convex and therefore represent hard problems for first-order algorithms.


\begin{example} {\rm {\bf (Strongly quasar-convex illustration)}} \label{exam:01}
In the first example, we consider a family of strongly quasar-convex functions of the type: $h: \mathbb{R}^{2} \rightarrow \mathbb{R}$, defined as:
\begin{align}\label{function:01}
h(x) = h_{1}(\lVert x \rVert) h_{2} \left( \frac{x}{\lVert x \rVert} \right),
\end{align}
where 
\begin{align*}
 & \hspace{0.4cm} h_{1} (\lVert (x_{1}, x_{2}) \rVert) = \max\{ q_{1} \lVert (x_{1}, x_{2}) \rVert^{2}, q_{2} \lVert  (x_{1}, x_{2})  \rVert^{2} - k\}, ~ q_{2} > q_{1} > 0, ~ k \in \mathbb{N}, \\
 & h_{2} \left( \frac{(x_{1}, x_{2})}{\lVert (x_{1}, x_{2}) \rVert} \right) = \frac{1}{4N} \sum_{i=1}^{N} \left( a_{i} \sin^{2}(b_{i} x_1) + c_{i} \cos^{2}(d_{i} x_2) \right) + 1,
\end{align*} 
with $N=10$, ${a_{i}}$ and ${c_{i}}$ independently and uniformly distributed on $[0, 20]$, while ${b_{i}}$ and ${d_{i}}$ are independently and uniformly distributed on $[-25, 25]$ for $h_{2}$.

Clearly, the minimum point of $h$ on $\mathbb{R}^{2}$ is $(x_{1}, x_{2}) = (0, 0)$, and $h$ is nonsmooth because $h_{1}$ is nonsmooth as the maximum of two functions. Furthermore, $h_{1}$ is strongly quasar-convex. Indeed, $h_{1}$ is strongly convex as the maximum of two strongly convex functions with modulus $\gamma_{1} = \min\{2 q_{1}, 2 q_{2}\} > 0$, thus $h_{1}$ is strongly quasar-convex with modulus $\kappa = 1$ and $\gamma = \gamma_{1} > 0$. Therefore, it follows from \cite[Proposition 8]{HADR} that $h = h_{1} h_{2}$ is strongly quasar-convex with the same modulus. However, note that Algorithm \ref{ppa:sqcx} does not require the exact values of $\kappa$ and $\gamma$ for being implemented.

Also we emphasize that $h$ is not strongly quasiconvex (not even quasiconvex) since its level sets are not convex, hence the PPA's from \cite{BC-2,Lara-9,rock-1976} and refe\-ren\-ces therein cannot be applied in this context. 

On the other hand, $h$ satisfies the assumptions of Theorem \ref{usual:prox}, i.e., the sequence generated by Algorithm \ref{ppa:sqcx} is minimizing and converges linearly to the optimal solution $\overline{x} = (0, 0)$.
\end{example}

\begin{example} {\rm {\bf (Quasar-convex illustration)}} \label{exam:02}
 In this second example, we consider a family of quasar-convex (not strongly, i.e., $\gamma = 0$) functions of the type: $h: \mathbb{R}^{2} \rightarrow \mathbb{R}$, defined as in \eqref{function:01},
where $h_{2}$ is given as in Example \ref{exam:01} while $h_{1}$ is given by
$$h_{1} (\lVert (x_{1}, x_{2}) \rVert) = \max\{ \lVert (x_{1}, x_{2}) \rVert, q \lVert (x_{1}, x_{2}) \rVert - k\}, ~ q > 1, ~ k \in \mathbb{N}.$$ 

Clearly, the minimum point of $h$ on $\mathbb{R}^{2}$ is $\overline{x} = (0, 0)$, a point in which $h$ is nonsmooth in virtue of $h_{1}$. Furthermore, since $h_{1}$ is convex, it is quasar-convex with modulus $\kappa = 1$ and $\gamma = 0$, thus using \cite[Proposition 8]{HADR} we get that $h$ is quasar-convex, too. An illustration of function $h$ is given in Figure \ref{fig:02} below. 
\begin{figure}[htbp]
\centering 
\includegraphics[width=1.02\linewidth]{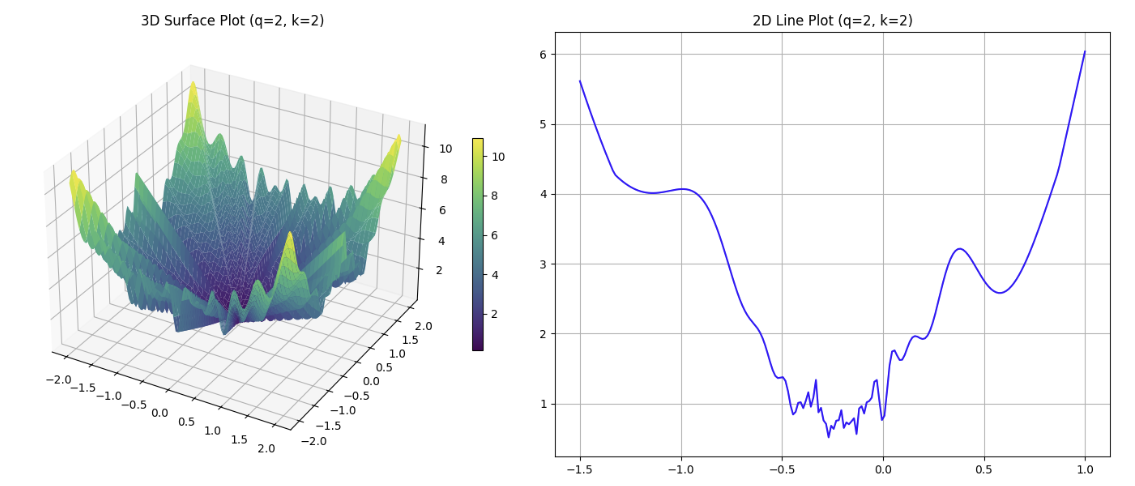} 
\caption{An illustration of the quasar-convex function $h(x)$ described in Exam\-ple \ref{exam:02} with $q=2$ and $k=2$. A 3D plot of the function $h(x)$ (left) and an arbitrary segment that does not contain the minimizer (right).
} \label{fig:02}
\end{figure}

As in the previous example, $h$ is neither convex nor strongly quasiconvex (not even quasiconvex) since its level sets are not convex (see Figure \ref{fig:02}), hence the PPA's from \cite{BC-2,Lara-9,rock-1976} and references therein cannot be applied.

On the other hand, $h$ satisfies the assumptions of  Theorem \ref{case gamma=0}, i.e., the sequence generated by Algorithm \ref{ppa:sqcx} is minimizing and converges to the optimal solution $\overline{x} = (0, 0)$.
\end{example}

\subsection{Implementations}
To solve the above two examples in this section via PPA, we cannot use the common smooth first-order method or second-order method for the subs\-pro\-blems since both subproblems involved in Examples \ref{exam:01} and \ref{exam:02} are nonsmooth. Therefore, we use the Clarke necessary optimality condition \cite[Proposition 2.4.11]{Clarke} on the subproblem in PPA which is (use also \cite[Corollary 1, page 39]{Clarke})
\begin{equation}\label{eq.subproblem_PPM}
 \overline{x} \in {\rm Prox}_{\beta_k h} (z) ~ \Longrightarrow ~ \frac{1}{\beta_{k}} (z - \overline{x}) \in \partial^{C} h(\overline{x}),
\end{equation}
where 
$\partial^{C}$ is the Clarke subdifferential \cite{Clarke}. Furthermore, since a piecewise smooth function is se\-mi\-smooth \cite[Proposition 2.26]{Ulbrich:2011} and the product of se\-mi\-smooth functions is still semismooth \cite[Proposition 1.75]{Solodov:2014}, we utilize the semismooth Newton method \cite{Qi:1993} for solving the subproblems \eqref{step:sqcx} in Algorithm \ref{ppa:sqcx}. Moreover, note that in the subproblems the product rule \cite[Proposition 2.3.13]{Clarke} holds as equality since the functions $h_{1}$ and $h_{2}$ (in both examples) are {\it regular} in the sense of \cite[Definition 2.3.4]{Clarke} (see \cite[Proposition 2.3.6]{Clarke}).



We compare Algorithm \ref{ppa:sqcx} with the semismooth Newton method in \cite{Qi:1993}, referred to as \textbf{PPA} and \textbf{SSN}, respectively. 
Our comparison primarily focuses on the optimal function values attained by the two methods. 
This choice is motivated by the fact that the semismooth Newton method is employed as a subroutine within PPA, so that both methods share a common computational core, making direct comparisons in terms of iteration counts and computational time less informative.

Nevertheless, for completeness, we also report the wall-clock time. The numerical results show that \textbf{PPA} exhibits significantly better robustness and convergence behavior than \textbf{SSN}. 
On the other hand, when both methods successfully converge, \textbf{PPA} typically requires more computational time, reflecting the additional cost of solving subproblems at each iteration.


The computational experiments were conducted on a machine equipped with an Intel Core i7-12700H CPU at 2.30GHz, 32 GB of RAM, and running Windows 10 Professional. All implementations were developed in MATLAB R2015b. 
The following conditions were used:
\begin{itemize}
    \item[(i)] We consider $N=2,5,10,20$ and for each value of $N$ we test {\it 50 instances
    } generated with different random parameters as described in Examples \ref{exam:01} and \ref{exam:02}. 
    \item[(ii)] The stopping criteria for PPA and SSN methods are
    \begin{align*}\label{}
     & \lVert x^{k+1} - x^k \rVert \leq \varepsilon ~~ {\rm or} ~~ k > 30000, \\
     & ~~ \lVert \partial^C h(x^k) \rVert \leq \varepsilon ~~ {\rm or} ~~ k > 30000,
    \end{align*}
    where $\varepsilon = 10^{-9}$, respectively.
    
    \item[(iii)] The inner loop of the SSN for solving the subproblem in PPA will stop if
    \begin{equation}
        \label{}
        \left \| \partial^C h(x^k) + \frac{1}{\beta_k} (x^k-x^{k-1}) \right \| \leq \varepsilon^{\prime} ~~ {\rm or}~~k>10000,
    \end{equation}
    where $\varepsilon^{\prime} < 10^{-7}$, $\beta_k = 0.05$ and $c=1$.
\end{itemize}

The obtained results are listed in the following two tables. 


\begin{table}[htbp]
\centering
\small %
\caption{Successful numbers (if the final function value is less than $10^{-6}$) and median of optimal values for PPA and SSN (Example \ref{exam:01}).}
\label{tab.1}
\begin{tabularx}{\textwidth}{l *{4}{>{\centering\arraybackslash}X}} %
\toprule
Method & \(N=2\) & \(N=5\) & \(N=10\) & \(N=20\) \\
\midrule
SSN & 19 / 1.34e-6 & 25 / 1.01e-6 & 31 / 7.41e-7 & 35 / 5.34e-7 \\
PPA    & 50 / 3.47e-10 & 50 / 3.86e-10 & 50 / 3.43e-10 & 50 / 3.13e-10 \\
\bottomrule
\end{tabularx}
\end{table}

\begin{table}[htbp]
\centering
\small %
\caption{Successful numbers (if the final function value is less than $10^{-3}$) and median of optimal values for PPA and SSN (Example \ref{exam:02}). We recall that \texttt{NaN} means {\it not a number}.}
\label{tab.2}
\begin{tabularx}{\textwidth}{l *{4}{>{\centering\arraybackslash}X}} %
\toprule
Method & \(N=2\) & \(N=5\) & \(N=10\) & \(N=20\) \\
\midrule
SSN & 5 / NaN & 2 / NaN & 8 / NaN & 4 / NaN \\
PPA & 49 / 4.63e-6 & 50 / 4.27e-6 & 50 / 3.57e-6 & 50 / 3.71e-6 \\
\bottomrule
\end{tabularx}
\end{table}
\begin{remark}
In Table \ref{tab.2}, we choose a less restrictive comparison standard for Exam\-ple \ref{exam:02} since the SSN hardly converges for most of the samples. \texttt{NaN} indi\-ca\-tes that the median cannot be computed since most of the runs did not converge.   
\end{remark}

\begin{table}[htbp]
\centering
\small %
\caption{Wall-clock time (s) of \textbf{PPA} and \textbf{SSN} for Example \ref{exam:01} (successful cases only).}
\label{tab.3}
\begin{tabularx}{\textwidth}{l *{4}{>{\centering\arraybackslash}X}} %
\toprule
Method & \(N=2\) & \(N=5\) & \(N=10\) & \(N=20\) \\
\midrule
SSN & 0.4784 & 0.5375 & 0.5319 & 0.5288 \\
PPA &  15.94 & 18.46 & 15.99 & 12.28 \\
\bottomrule
\end{tabularx}
\end{table}

\begin{table}[htbp]
\centering
\small %
\caption{Empirical linear contraction ratios $R_i$ and $R_f$ for PPA, reported at iterations k = 1, 2, 3, 4, 5, 6, 7 and 8. The values remain uniformly below 1, corroborating linear convergence of PPA for Example \ref{exam:01}.}
\label{tab.4}
\begin{tabularx}{\textwidth}{l *{8}{>{\centering\arraybackslash}X}} %
\toprule
Iteration & \(1\) & \(2\) & \(3\) & \(4\) & \(5\) & \(6\) & \(7\)& \(8\) \\
\midrule
$R_i$ & 0.1535 & 0.1721 & 0.2758 & 0.3713 & 0.3720 & 0.3720 & 0.3720 & 0.3720 \\
$R_f$ &  0.0332 & 0.0272 & 0.0290 & 0.1487 & 0.1384 & 0.1384 & 0.1384 & 0.1384 \\
\bottomrule
\end{tabularx}
\end{table}

\begin{figure}[h!]
    \centering
    \begin{subfigure}[b]{0.45\textwidth}
        \centering
        \includegraphics[width=\textwidth]{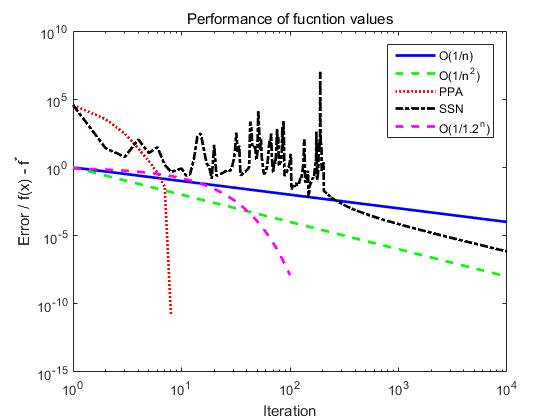}
        \caption{Function value sequence}
        \label{fig:ex27_fvalue}
    \end{subfigure}
    \hfill
    \begin{subfigure}[b]{0.45\textwidth}
        \centering
        \includegraphics[width=\textwidth]{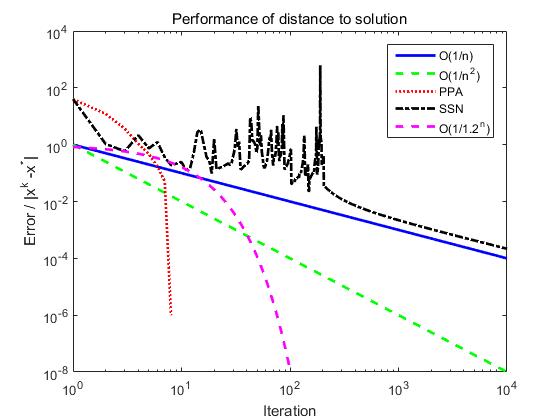}
        \caption{Norm of iterates}
        \label{fig:ex27_xknorm}
    \end{subfigure}

    \caption{Convergence of function value sequence and the norm of iterates over iterations for Example~\ref{exam:01}.}
    \label{fig:ex27_convergence}
\end{figure}

\begin{figure}[htbp]
    \centering
    \begin{subfigure}[b]{0.45\textwidth}
        \centering
        \includegraphics[width=\textwidth]{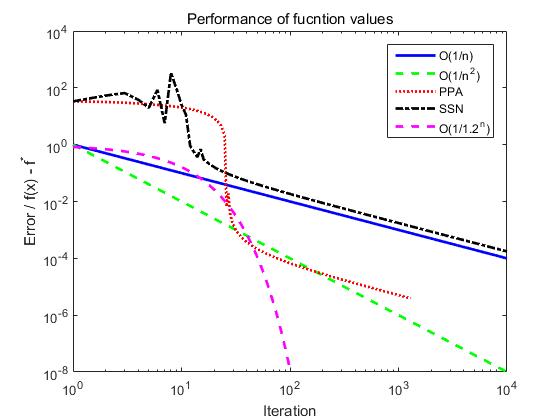}
        \caption{Function value sequence}
        \label{fig:ex28_fvalue}
    \end{subfigure}
    \hfill
    \begin{subfigure}[b]{0.45\textwidth}
        \centering
        \includegraphics[width=\textwidth]{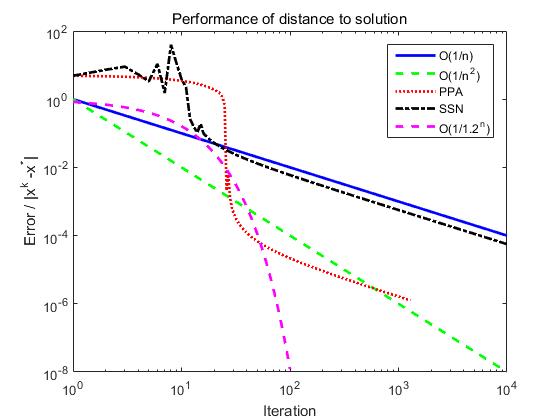}
        \caption{Norm of iterates}
        \label{fig:ex28_xknorm}
    \end{subfigure}

    \caption{Convergence of function value sequence and the norm of iterates over iterations for Example~\ref{exam:02}.}
    \label{fig:ex28_convergence}
\end{figure}
\begin{remark}
In several examples, we observe that SSN converges with os\-ci\-lla\-to\-ry behavior while PPA does not (see, for instance, Figures \ref{fig:ex27_convergence} and \ref{fig:ex28_convergence} which correspond to one instance selected from our experiments of Examples \ref{exam:01} and \ref{exam:02}, respectively).
    
\end{remark}

According to the numerical experiments, we conclude the following:
\begin{enumerate}
    \item In Example \ref{exam:01}, PPA and SSN converges in all instances but PPA achieves higher accuracy (Table \ref{tab.1}), while in Example \ref{exam:02} only PPA converges since SSN diverges in many cases (Table \ref{tab.2}). Confirming the advantages of PPA by solving subproblems instead of attacking directly the original problem. 

    \item PPA exhibits linear convergence in the tested instances of Example \ref{exam:01} for the strongly quasar-convex function (see Figure \ref{fig:ex27_convergence}) while for a general quasar-convex function (i.e., not strongly), PPA converges only sublinearly (see Figure \ref{fig:ex28_convergence}), as expected from our theoretical results of Section \ref{sec:04}. 

    \item For completeness, we report the wall-clock time of \textbf{PPA} and \textbf{SSN} for selected instances in Example \ref{exam:01}. This comparison is restricted to these instances, since \textbf{SSN} fails to converge for many cases in Example \ref{exam:02}, as indicated in Table \ref{tab.2}. When both methods successfully converge, Table \ref{tab.3} shows that \textbf{PPA} is considerably more time-consuming than \textbf{SSN}. This is of course expected, as each iteration of \textbf{PPA} requires solving a subproblem via the semismooth Newton method, leading to a higher overall computational cost despite its improved robustness.
        
    

    \item Table \ref{tab.4} reports the empirical ratios
    \[
        R_f(k) =:\frac{f_{k+1}}{f_k},~~~~R_i(k) = \frac{\|x^{k+1} - x^*\|}{\|x^k - x^*\|},
    \]
    generated by PPA for Example \ref{exam:01}. All ratios stabilize strictly below one, providing clear numerical evidence of linear convergence of the function values and iterates, in agreement with the theoretical predictions for strongly quasar-convex function.
\end{enumerate}

\section{Conclusions}\label{sec:06}

We contributed to the extension of the remarkable properties of PPA from (strongly) convex functions to a broader class of functions, known as (strongly) quasar-convex. By generalizing the standard properties of proximity operators, we prove that PPA generates a minimizing sequence for quasar-convex functions and converges linearly in the strongly quasar-convex case, achieving a rate matching that of the strongly convex case.  Additionally, we present several noteworthy examples of nonsmooth (strongly) quasar-convex functions, which have practical applications in machine learning, microeconomic theory, and beyond.

Our developments open the door to new research directions, as algorithms for nonsmooth quasar-convex optimization remain largely unexplored. A natural next step would be to extend these results to proximal-gradient (splitting) methods \cite{ABS} and high-order PPA variants \cite{AIKL}. We leave these promising extensions for future work.



\section{Declarations}\label{sec:06}







\subsection{Availability of Supporting Data}

No datasets were generated during the present study. The Matlab code for the numerical experiments is available upon reasonable request from all authors.

%

\subsection{Competing Interests}

There are no conflicts of interest or competing interests related to this manuscript.

\subsection{Funding}

This research was partially supported by ANID--Chile under project Fondecyt Regular 1241040 (Lara) and Conselho Nacional de Desenvolvimento Cient\'ifico e Tecnol\'ogico CNPq 153172/2024-0 (Liu).


\end{document}